\documentclass[11pt, twoside, a4paper]{article}

\usepackage{amsmath}
\usepackage{amssymb}
\usepackage{amsthm}
\usepackage{graphicx}

\usepackage{textcomp}
\usepackage[T1]{fontenc}
\usepackage{marvosym}
\usepackage[sc]{mathpazo}

\usepackage{authblk}
\usepackage[usenames]{xcolor}
\usepackage{lastpage}

\usepackage{hyperref}

\definecolor{ForestGreen}{rgb}{0.15,0.416,0.18}
\definecolor{EgyptBlue}{rgb}{0.063,0.2,0.65}
\hypersetup{
	colorlinks=true,
	linkcolor=EgyptBlue,         
   citecolor=EgyptBlue,          
   urlcolor=ForestGreen          
}

\newtheorem{theorem}{Theorem}[section]
\newtheorem{corollary}[theorem]{Corollary}
\newtheorem{lemma}[theorem]{Lemma}
\newtheorem{proposition}[theorem]{Proposition}

\theoremstyle{definition}

\theoremstyle{definition}
\newtheorem{remark}[theorem]{Remark}
\theoremstyle{definition}
\newtheorem{example}[theorem]{Example}

\numberwithin{equation}{section}
\numberwithin{table}{section}
\numberwithin{figure}{section}

\title{Rational limit cycles of Abel differential equations}

\author[1]{\textbf{Jos\'e Luis Bravo}}
\author[2]{\textbf{Luis \'Angel Calder\'on}\footnote{Corresponding author. Email: l.calderon@uib.es}}
\author[1]{\textbf{Ignacio Ojeda}}

\affil[1]{Departamento de Matematicas, Universidad de Extremadura, 06071 Badajoz, Spain}
\affil[2]{Departament de Ciencies Matematiques i Informatica, IAC3 Institute of Applied Computing \& Community Code, Universitat de les Illes Ballears, 07122 Palma, Spain}

\newcommand{\doi}[1]{\url{https://doi.org/#1}}
\newcommand{\MR}[1]{\href{https://www.ams.org/mathscinet-getitem?mr=#1}{MR#1}}
\newcommand{\ZB}[1]{\href{https://zbmath.org/?q=an:#1}{Zbl~#1}}

\makeatletter
\renewcommand{\maketitle}{\bgroup\setlength{\parindent}{0pt}

\vspace{1truecm}
\begin{center}{\vbox{\titlefont\@title}}\end{center}
\vspace{0.5truecm}
\begin{center}{\@author} \end{center}

\egroup
}

\renewcommand{\@fnsymbol}[1]{%
    \ifcase#1 \or {\,\Letter\!} \or\textasteriskcentered\or \textasteriskcentered\textasteriskcentered 
    \else\@ctrerr\fi}
\makeatother

\newcommand*{\titlefont}{\fontsize{18}{21.6}\selectfont\textbf}

\begin{document}

\maketitle

\pagestyle{plain}

\begin{center}
\noindent
\begin{minipage}{0.85\textwidth}\parindent=15.5pt

\bigskip

{\small{
\noindent {\bf Abstract.} We study the number of rational limit cycles of the Abel equation $x'=A(t)x^3+B(t)x^2$, where $A(t)$ and $B(t)$ are real trigonometric polynomials. We show that this number is at most the degree of $A(t)$ plus one.}
\smallskip

\smallskip

}

\end{minipage}
\end{center}

\section{Introduction}

The Abel differential equation 
\[	
x'=A(t)x^3+B(t)x^2+C(t)x,
\]
where $A(t),B(t)$ and $C(t)$ are trigonometric polynomials has been studied by many authors, either for its relation to higher degree phenomena (see e.g.~\cite{LN}), for applications to real-world models (see e.g.~\cite{BNP}), or for its own intrinsic interest (see e.g. \cite{G}).

In this paper we consider Abel differential equations without linear term, that is, 
\begin{equation}	
x'=A(t)x^3+B(t)x^2,\label{eq:Abel}
\end{equation}
with $A(t)$ and $B(t)$ being real trigonometric polynomials.

Among the main problems related to this equation, we could name the Smale-Pugh~\cite{S} problem, which is considered as a particular case of the 16th Hilbert problem. The problem consists in bounding the number of limit cycles of \eqref{eq:Abel}, that is, the number of isolated periodic solutions in the set of periodic solutions of the equation. In connection with this problem, Lins-Neto~\cite{LN} proved that there is no upper bound on the number of limit cycles of~\eqref{eq:Abel}.

Another important problem often mentioned in the literature is the Poin\-car\'e center-focus problem applied to this setting. Trivially, $x(t) = 0$ is a solution of the equation. The problem asks when the equation~\ref{eq:Abel} has a center at $x(t) = 0$, i.e., all solutions in a neighborhood of the solution $x(t) = 0$ are closed. This problem for \eqref{eq:Abel} was proposed by Briskin, Fran\c{c}oise, and Yondim~\cite{BFY1,BFY2}. There are some progresses in solving this problem, for example, it has been proved that the composition condition determine certain parametric centers when the coefficients are polynomials~\cite{PSCC}, but the problem is open for trigonometric coefficients. Moreover, it is conjectured that the converse is true when $A(t)$ and $B(t)$ are polynomials, with some relevant evidences that support it. For more more details, the reader may refer to~\cite{GGS16,G}.

When faced with the Smale-Pugh problem, one of the most common strategies for obtaining upper bounds on the number of limit cycles is to impose some additional restrictions on $A(t)$ and $B(t)$. For example, it has been shown that if $A(t)\neq0$ or $B(t)\neq0$ does not change sign, or if there exist $\alpha,\beta\in\mathbb{R}$ such that $\alpha A(t)+\beta B(t)\neq0$ does not change sign, then the equation has at most three limit cycles \cite{GLL,AGG}. See \cite{G} for more information.

Another strategy is to focus on the problem for limit cycles of a certain form or with certain properties that may be of particular interest. An example of such a result is that the generalized Abel equation with polynomial coefficients and degree $n$ has at most $n$ polynomial limit cycles, see \cite{GGL13}.

Also, the number of rational solutions of \eqref{eq:Abel} when $A(t),B(t)$ are polynomials has been profusely studied. For instance, in \cite{QSY21} the authors obtain upper bounds on the number of rational periodic solutions of \eqref{eq:Abel} under certain conditions on the degrees of $A(t),B(t)$, and in \cite{BCFO}, a general, non-optimal upper bound has been obtained. It has been also studied for $A(t),B(t)$ trigonometric polynomials, in \cite{V22}. Note that the rational solutions are not necessarily limit cycles, as they may be part of a centre.

In this paper we obtain an upper bound on the number of rational limit cycles of the equation \eqref{eq:Abel}, i.e.,  limit cycles of the form $x(t)=Q(t)/P(t)$, where $P(t)$ and $Q(t)$ are real trigonometric polynomials and $P(t) \neq 0$ for all $t \in \mathbb{R}$. Recall that a real trigonometric polynomial of degree $n$ is an expression of the form \[\sum_{k=0}^n a_k \cos(kt) + b_k \sin(kt),\quad a_k, b_k \in \mathbb{R}\] with $a_n \cdot b_n \neq 0$. As usual, we write $\mathbb{R}[\cos(t), \sin(t)]$ for the ring of real trigonometric polynomials.

Using this notation, our main result is as follows.

\begin{theorem}\label{th:main}
Let $A(t), B(t)\in \mathbb{R}[\cos(t), \sin(t)]$. If the degree of $A(t)$ is odd or less than twice the degree of $B(t)$, then \eqref{eq:Abel} has at most two non-trivial rational limit cycles. Otherwise, the number of non-trivial rational limit cycles of equation \eqref{eq:Abel} is at most the degree of $A(t)$ plus one.
\end{theorem}

To prove this result, we consider each rational limit cycle $x(t)= Q(t)/P(t)$ as an invariant trigonometric algebraic curve of degree one in $x$ with real trigonometric coefficients (Proposition \ref{prop:key}), that is, an invariant curve of \eqref{eq:Abel} of the form $Q(t)-P(t)x=0$, where $P(t)$ and $Q(t)$ are real trigonometric polynomials and $P(t) \neq 0$ for all $t \in \mathbb{R}$. Therefore, to bound the number of limit rational cycles, we bound the number of invariant algebraic curves of degree one in $x$ with real trigonometric coefficients of \eqref{eq:Abel} such that \eqref{eq:Abel} has no center at the origin. In particular, to prove the second part of Theorem~\ref{th:main}, we bound the maximum number of invariant curves of this type such that \eqref{eq:Abel} does not have a Darboux first integral. 

The study of invariant curves of degree one in $x$ is interesting in itself, regardless of whether they correspond to rational limit cycles or not. Therefore, we include a method to parameterize the Abel equations~\ref{eq:Abel} that have at least two non-trivial invariant algebraic curves of degree one in $x$ with real trigonometric coefficients.

The structure of the paper is as follows. In the first section we characterize the invariant algebraic curves of degree one in $x$ with real trigonometric coefficients such that the equation \eqref{eq:Abel} has and we show some of their properties. Then we prove the first part of Theorem \ref{th:main} and give the parameterization mentioned above. In the second section, we apply Darboux's integrability theory to this situation, which allows us to complete the proof of the main result of the paper.

Due to the fact that $\mathbb{R}[\cos(t),\sin(t)]$ is not a unique factorization domain, some necessary facts and results about factorization in these rings are collected in Appendix \ref{apendice}. 

\section{Real trigonometric invariant algebraic curves of degree one in x}

Consider the Abel equation \eqref{eq:Abel}, set $g(t,x):=A(t)x^3+B(t)x^2$ and denote the associated vector field by $\mathcal{X}$, that is,
\[\mathcal{X}=\frac{\partial}{\partial t}+g\frac{\partial}{\partial x}.\]

Fixed $f\in\mathcal{C}^1$, the curve $f(t,x)=0$ is said to be an invariant curve of \eqref{eq:Abel} if there exists $K\in\mathcal{C}^0$, called the cofactor of $f(t,x)$, such that 
$$
\left(\mathcal{X}f\right)(t,x)=\left(\frac{\partial f}{\partial t}+g\frac{\partial f}{\partial x}\right)(t,x)=K(t,x)f(t,x).
$$

Note that $f_0(t,x):=x=0$ is always an invariant curve of \eqref{eq:Abel} with the cofactor $K_0(t,x)=A(t)x^2+B(t)x$. 

If $f(t,x)=0$ is invariant and $x(t)$ is a solution of \eqref{eq:Abel}, then, for any $t_0$ in the domain of the solution,
\[f(t,x(t))=f(t_0,x(t_0))\exp\left(\int_{t_0}^{t} K(s,x(s))ds\right).\]	
Therefore, if $f(t_0,x(t_0))=0$, then $f(t,x(t))=0$ for all $t$. Consequently $f(t,x)=0$ consists of trajectories of solutions of the equation.

From now on, unless otherwise stated, $A(t)$ and $B(t)$ are real trigonometric polynomials, that is, $A(t), B(t) \in \mathbb{R}[\cos(t), \sin(t)]$. 

The first objective of this section is to characterize the invariant algebraic curves $f(t,x) = 0$ of \eqref{eq:Abel} such that $f(t,x) = Q(t) - P(t) x \in \mathbb{R}[\cos(t), \sin(t)][x]$ with $P(t) \neq 0$, for all $t \in \mathbb{R}$. 

Let $P(t)$ and $Q(t)$ be real trigonometric polynomials with $P(t) \neq 0,$ for all $t \in \mathbb{R}$. If $R(t) \in \mathbb{R}[\cos(t),\sin(t)]$ is a common factor of $P(t)$ and $Q(t)$, meaning that there exist factorizations of $P(t)$ and $Q(t)$ in which $R(t)$ appears (see Appendix \ref{apendice} for details), we have that $Q(t) - P(t) x=0$ is an invariant curve of \eqref{eq:Abel} if and only if $Q(t)/R(t) - (P(t)/R(t)) x=0$ is an invariant curve of \eqref{eq:Abel}. Thus, in what follows, we always assume that $P(t)$ and $Q(t)$ have no common factors, that is, $Q(t) - P(t) x$ is irreducible in $\mathbb{R}[\cos(t),\sin(t)][x]$ and, by Corollary \ref{cor:nofactors}, in $\mathbb{C}[\cos(t),\sin(t)][x]$.

\begin{remark}\label{Rem:invcurve}
To simplify the exposition from now on, we simply say invariant curves of \eqref{eq:Abel} to refer to invariant curves of \eqref{eq:Abel} of the form $Q(t) - P(t) x = 0$, where $P(t) \neq 0,$ for all $t \in \mathbb{R}$, and $Q(t) - P(t) x$ is irreducible.
\end{remark}

The following result establishes the relationship between rational limit cycles $x(t)= Q(t)/P(t)$ and invariant curves of \eqref{eq:Abel}.

\begin{proposition}\label{prop:key}
Let $P(t)$ and $Q(t)$ be real trigonometric polynomials. If $P(t) \neq 0$ for all $t \in \mathbb{R}$, then $x(t)=Q(t)/P(t)$ is a solution of \eqref{eq:Abel} if and only if $Q(t) - P(t) x=0$ is an invariant curve of \eqref{eq:Abel}.
\end{proposition}

\begin{proof}
Since $\mathbb{R}[\cos(t), \sin(t)]$ is a domain (see \cite[Theorem 3.1]{P-PH}) its field of fractions, $\Sigma$, is well-defined. Thus, we can perform the Euclidean division of $(\mathcal{X} f)(t,x) = Q'(t)-P'(t) x-P(t)\Big(A(t)x^3+B(t)x^2\Big)$ by $f(t,x):=Q(t) - P(t) x$ in $\Sigma[x]$, so that there exist unique $K(t,x)$ and $Z(t,x) \in \Sigma[x]$ such that \[(\mathcal{X} f)(t,x) = K(t,x)(Q(t) - P(t) x) + Z(t).\] Concretely, 
\[Z(t) = P(t) \left( \left(\frac{Q(t)}{P(t)} \right)' - A(t) \left(\frac{Q(t)}{P(t)} \right)^3 - B(t) \left(\frac{Q(t)}{P(t)} \right)^2 \right) \]
and \begin{align*} 
K(t,x) = & A(t) x^2 + \left(A(t) \left(\frac{Q(t)}{P(t)} \right) + B(t) \right) x\\ & + A(t) \left(\frac{Q(t)}{P(t)} \right)^2 + B(t) \left(\frac{Q(t)}{P(t)} \right) + \frac{P'(t)}{P(t)}.
\end{align*} 
Note that $K(t,x) \in \mathcal{C}^0$ because $P(t) \neq 0$ for all $t \in \mathbb{R}$. Therefore, we conclude that the necessary and sufficient condition for $Q(t) - P(t) x = 0$ to be an invariant curve of \eqref{eq:Abel} is $Z(t) = 0$, which, given the expression of $Z(t)$, is equivalent to $x(t)=Q(t)/P(t)$ being a solution of \eqref{eq:Abel}.
\end{proof}

Now, our approach is as follows: instead of directly bounding the number of rational limit cycles of \eqref{eq:Abel}, we bound the number of invariant curves of degree one in $x$ such that \eqref{eq:Abel} has not a center. According to Proposition \ref{prop:key}, this will be an upper bound on the number of rational limit cycles. 

Next we give a condition so that $Q(t)-P(t)x=0$ is an invariant curve of \eqref{eq:Abel}. But first we need a lemma.

\begin{lemma}\label{lema1}
Let $P(t)$ and $Q(t)$ be real trigonometric polynomials with $P(t) \neq 0$ for all $t \in \mathbb{R}$. If $Q(t) - P(t) x = 0$ is an invariant curve of \eqref{eq:Abel}, then the corresponding cofactor is a polynomial in $x$ with real trigonometric polynomial coefficients. 
\end{lemma}

\begin{proof}
Let $f(t,x):=Q(t) - P(t) x$. Arguing as in the proof of Proposition \ref{prop:key}, we obtain that there exists $\tilde K(t,x) \in \mathbb{R}[\cos(t), \sin(t)][x]$ such that $P(t)^2 (\mathcal Xf)(t,x) = \tilde K(t,x) f(t,x)$. 

By Corollary \ref{cor:nofactors}, $f(t,x)$ is irreducible in $\mathbb{C}[\cos(t), \sin(t)][x]$. Moreover, as $\mathbb{C}[\cos(t), \sin(t)]$ is an Euclidean domain (see \cite[Theorem 2.1]{P-PH}), it is an unique factorization domain and therefore $\mathbb{C}[\cos(t), \sin(t)][x]$ is also an unique factorization domain. Thus, we have that $P(t)^2$ or $ (\mathcal Xf)(t,x)$ are divisible by $f(t,x)$ which necessarily implies that there exists $H(t,x) \in \mathbb{C}[\cos(t),\sin(t)][x]$ such that $(\mathcal Xf)(t,x) = H(t,x) f(t,x)$ for degree reasons. 

Finally, since both $(\mathcal Xf)(t,x)$ and $f(t,x)$ are polynomials in $x$ with real trigonometric polynomial coefficients, we conclude that $H(t,x)$ is also a polynomial in $x$ with real trigonometric polynomial coefficients.
\end{proof}

The next result gives the condition for $Q(t)-P(t)x=0$ to be an invariant curve of \eqref{eq:Abel}. This result is proved in \cite{LLWW} for the polynomial case. Moreover, since rational solutions are equivalent to the invariant curves of degree one in $x$, as mentioned in the introduction, the following result can also be obtained from those of \cite{V22} for the trigonometric case. Here, we provide a simplified proof.

\begin{proposition}\label{prop:inv} 
The curve $Q(t)-P(t)x=0$ is an invariant curve of \eqref{eq:Abel} if and only if $Q(t) = c \in \mathbb{R}$ and there exists a trigonometric polynomial $R(t)$ such that \[A(t)=(P(t)/c) R(t), \quad B(t)=-P'(t)/c-R(t).\] In this case, the corresponding cofactor is equal to $A(t) x^2 - (P'(t)/c) x$.
\end{proposition}

\begin{proof}
Let $f(t,x):=Q(t)+P(t)x=0$ be an invariant curve of \eqref{eq:Abel}. Arguing as in the proof of Proposition \ref{prop:key} and taking advantage of the fact that $x(t) = Q(t)/P(t)$ is a solution of \eqref{eq:Abel}, we have that the corresponding cofactor can be written as
\[
K(t,x) = A(t) x^2 - \left(\frac{P(t)}{Q(t)}\right)' x + \frac{Q'(t)}{Q(t)}.
\]
Since, by Lemma \ref{lema1}, $K(t,x) \in \mathbb{R}[\cos(t),\sin(t) ][x]$, we have in particular that $Q'(t)$ is equal to $K_0(t) Q(t)$ for some $K_0(t) \in \mathbb{R}[\cos(t),\sin(t)]$. Comparing degrees,  either $Q(t) = 0$ or $K_0$ is constant. In the latter case, \[\vert Q(t) \vert = e^{\int K_0 dt}.\] Thus we conclude that $K_0=0$, and $Q(t) = c \in \mathbb{R}$. Note that, in this case, $K(t,x) = A(t) x^2 - (P'(t)/c) x$.

Finally, noting that 
\[-\frac{P'(t)}{c} = A(t) \left(\frac{c}{P(t)} \right) + B(t)\]
we conclude that $B(t) = -P'(t)/c - R(t)$ where $R(t) = c A(t)/P(t) \in \mathbb{R}[\cos(t),\sin(t)]$.

The converse follows by direct checking.
\end{proof}

As mentioned in the proof of Proposition~\ref{prop:inv}, the curve $c-P(t)x=0$ is an invariant curve of the equation \eqref{eq:Abel} if and only if
\[-\frac{P'(t)}{c} = A(t) \left(\frac{c}{P(t)} \right) + B(t).\] Without loss of generality, we can assume $c=1$, so that
\begin{equation}
P(t)P'(t)+P(t)B(t)+A(t)=0.\label{CondInv}
\end{equation}

Note that if equation \eqref{eq:Abel} has an invariant curve of the form $1-Kx=0$ with $K$ a non-zero constant, then the Abel equation becomes the separated variable equation $x'=B(t)x^2(-Kx+1)$ with constant solutions $0$ and $1/K$. If $\int_0^{2\pi} B(t)\,dt\neq 0$ these constant solutions are the unique limit cycles, while if $\int_0^{2\pi} B(t)\,dt= 0$ every bounded solution is periodic, so it has no limit cycles. Hence, we consider only the case $\deg(P)\geq 1$.

We will say that an invariant curve $1-P(t)x=0$ has degree $n$ if $\deg(P)=n$. Next we prove that the sum of the degrees of two invariant curves is the degree of $A$.

\begin{proposition}\label{prop2c}
If $1 - P_1(t) x =0$ and $1 - P_2(t) x = 0$ are two different invariant curves of \eqref{eq:Abel}, then $\deg(P_1)+\deg(P_2)=\deg(A)$. Consequently, if $\deg(P_1) = \deg(P_2),$ then $\deg(P_1) = \deg(P_2) = \deg(A)/2$.
\end{proposition}
 
\begin{proof}
By Proposition \ref{prop:inv}, there exist trigonometric polynomials $R_1(t)$ and $R_2(t)$ such that $P_1(t)R_1(t) = A(t) = P_2(t)R_2(t)$ and $-P'_1(t) - R_1(t) = B(t) =  -P'_2(t) - R_2(t)$. Thus, \[P_1(t)(P'_2(t)+R_2(t)-P'_1(t)) = A(t) = P_2(t)R_2(t).\]
Therefore 
\[P_1(t)(P_2(t)-P_1(t))' = P_1(t)(P'_2(t)-P'_1(t)) = R_2(t) (P_2(t)-P_1(t)).\]
Now, since $\deg(P_2(t)-P_1(t)) = \deg((P_2(t)-P_1(t))') $, we conclude that $\deg(P_1) = \deg(R_2) = \deg(A)-\deg(P_2)$, from which our claim follows. 
\end{proof}

The following example shows that \eqref{eq:Abel} can have two limit cycles of different degrees. 

\begin{example}\label{ejemplo1}
Let $P_1(t) = 2 (\cos(t)+2)(\sin(t)+2)$ and $P_2(t)=(\cos(t)+2)(\sin(t)+2)(\sin(t)+4)$. By Proposition \ref{prop:inv}, we have that $1-P_i(t) x,\ i = 1,2$ are invariant curves of \eqref{eq:Abel} for 
\begin{align*}
    A(t) = (\cos(t)+2)(\sin(t)+2)(\sin(t)+4)(3\cos(2t)+8\cos(t)-4\sin(t)+1)
\end{align*}
and 
\begin{align*}
B(t) = -3/4 \sin(3t)-9\cos(2t)-2 \sin(2t)-20\cos(t)+49/4 \sin(t) -1.
\end{align*}
Note that $5=\deg(A) = \deg(P_1) + \deg(P_2)=2+3$. Moreover, since $\int_0^{2\pi} B(t)\,dt = -2 \pi \neq 0$, \eqref{eq:Abel} does not have a center (see, for instance, \cite[Lemma~7]{AGG}), so the solutions $x(t)=1/P_1(t)$ and $x(t)=1/P_2(t)$ are limit cycles.
\end{example}

Remember that $x(t) = 0$ is always an invariant curve of \eqref{eq:Abel}. It corresponds to the case $Q(t) = 0$ and we call it a trivial invariant curve.

\begin{corollary}\label{coro3c}
If equation \eqref{eq:Abel} has three or more non-trivial invariant curves, then they all have degree $\deg(A)/2$.
\end{corollary}

\begin{proof}
Suppose that equation \eqref{eq:Abel} has three invariant curves $1-P_1(t)x=0$, $1-P_2(t)x=0$ and $1-P_3(t)x=0$. Then, by Proposition \ref{prop2c}, $\deg(P_1)+\deg(P_2)=\deg(P_1)+\deg(P_3)=\deg(P_2)+\deg(P_3)=\deg(A)$, which implies $\deg(P_1)=\deg(P_2)=\deg(P_3)=\deg(A)/2$. 
\end{proof}

Now, it is easy to give two conditions for equation \eqref{eq:Abel} to have at most two non-trivial invariant curves, which proves the first part of the main theorem (Theorem \ref{th:main}). 

\begin{corollary}\label{corosuf2c}
If $\deg(A)$ is odd or if $\deg(B)>\deg(A)/2$, then \eqref{eq:Abel} has has at most two non-trivial invariant curves.
\end{corollary} 

\begin{proof}
If $\deg(A)$ is odd the claim follows directly from Corollary \ref{coro3c}. 

So, suppose that $\deg(A)$ is even and $\deg(B) > \deg(A)/2$. If $1-P(t)x=0$ is an invariant curve of \eqref{eq:Abel}, then by Proposition \ref{prop:inv} there exists $R(t) \in \mathbb{R}[\cos(t),\sin(t)]$ such that $A(t) = P(t)R(t)$ and $B(t) = -P'(t) - R(t)$, then $\deg(B)\leq\max\{\deg(P) = \deg(P'),\deg(A)-\deg(P)\}$. Thus if $\deg(P)=\deg(A)/2$, then $\deg(B)\leq\deg(A)/2$, contradicting the hypothesis. This fact, together with Proposition~\ref{prop2c} and Corollary~\ref{coro3c}, completes the proof of the claim. 
\end{proof}

We can now write this last result in terms of rational limit cycles.

\begin{corollary}\label{corosuf2cl}
If  $\deg(A)$ is odd or if $\deg(B)>\deg(A)/2$, then \eqref{eq:Abel} has has at most two non-trivial rational limit cycles.
\end{corollary} 

In \cite{BCFO} a parameterization is given for all cases of equations $x' = A(t) x^3 + B(t) x^2,$ with $A(t), B(t) \in \mathbb{C}[t]$ which have at least two non-trivial polynomial polynomial invariant curves. Before we finish proving the main result of the paper in the next section, let us see that a similar parametrization of the rational limit cycles can be obtained in this case.

\begin{proposition}\label{prop:param}
Equation \ref{eq:Abel} has two different non-trivial invariant cur\-ves if and only if there exist $G(t), \hat G(t),S_1(t) \in \mathbb{R}[\cos(t),\sin(t)]$ and $k \in \mathbb{R} \setminus \{0\}$, such that $G(t),S_1(t),S_1(t)+k\hat G(t) \neq 0$ for all $t \in \mathbb{R}$, every irreducible factor of $\hat G(t)$ divides $G(t)$, and the functions $A,B$ satisfy \[A(t) = G(t) S_1(t) (S_1(t)+k \hat G(t)) \left(G'(t) + \frac{G(t) \hat G'(t)}{\hat G(t)}\right),\] \[B(t) = -(G(t) S_1(t))' - (S_1(t)+k \hat G(t)) \left(G'(t) + \frac{G(t) \hat G'(t)}{\hat G(t)}\right),\]  Furthermore, in this case the two different invariant curves are \[1 -  G(t) S_1(t) x = 0\quad\text{and}\quad 1 -  G(t) (S_1(t)+k \hat G(t) ) x = 0.\]
\end{proposition}

\begin{proof}
Assume $1-P_1(t)x=0$, $1-P_2(t)x=0$ are two different non-trivial invariant curves of \eqref{eq:Abel}. Recall that by Remark~\ref{Rem:invcurve}, $P_1(t) \neq 0$ and $P_2(t) \neq 0,$ for all $t \in \mathbb{R}$, so they have unique decomposition (see Corollary \ref{cor:ufd}). Thus there exists their greatest common divisor. Set $G(t) := \gcd(P_1(t), P_2(t))$,  $S_1(t) := P_1(t)/G(t)$ and $S_2(t):=P_2(t)/G(t)$. Moreover, since, by Proposition \ref{prop:inv}, $P_1(t)$ and $P_2(t)$ divide $A(t)$, there exists $S(t) \in \mathbb{R}[\cos(t),\sin(t)]$ such that \[A(t) = G(t) S_1(t) S_2(t) S(t)\] and, by Proposition \ref{prop:inv}, \[B(t) = -(G(t) S_1(t))' - S_2(t) S(t) = -(G(t) S_2(t))' - S_1(t) S(t).\] Thus, \[\Big(G(t)(S_2(t) - S_1(t))\Big)' = (S_2(t) - S_1(t)) S(t).\] Let $\hat G(t)$ be the product of all the factors of $S_2(t) - S_1(t)$ that divide $G(t)$; note that $\hat G(t)$ is well-defined by Proposition \ref{Prop P0} because $G(t) \neq 0$ for all $t \in \mathbb{R}$. Set $H(t) := (S_2(t)-S_1(t))/\hat G(t). H(t)$ does not necessarily have a unique decomposition; however, by construction, no irreducible factor of $H(t)$ (in any of its factorizations) can divide $G(t)$. Now, from  
\begin{align*}
G'(t) H(t) \hat G(t) + G(t) (H(t) \hat G(t))' & = (G(t)(S_2(t) - S_1(t))) \\ & = (S_2(t) - S_1(t)) S(t) \\ & = H(t) \hat G(t) S(t),
\end{align*}
it follows that $G(t)H'(t)\hat G(t) + G(t)H(t)\hat G'(t) = H(t) \hat G(t) (S(t)-G'(t))$. So, \begin{equation}\label{ecu:S} G(t)H'(t)\hat G(t) = H(t) \Big(\hat G(t) (S(t) - G'(t)) - G(t) \hat G'(t)\Big).\end{equation} Therefore, since $G(t) \hat G(t)$ have no real zeros and no common irreducible factors with $H(t)$, by Corollary \ref{cor:ufd2}, $G(t) \hat G(t)$ divides $R(t) := \hat G(t) (S(t)-G'(t)) - G(t) \hat G'(t)$. Moreover, noticing that $\deg(R(t)) \leq \deg(\hat G(t) G'(t)) = \deg(\hat G(t) G(t))$, we have that $R(t)/(G(t) \hat G(t)) = k_0 \in \mathbb{R}[\cos(t), \sin(t)]$. Therefore, $H'(t) = H(t) R(t)/(\hat G(t) G(t)) = H(t) k_0$ and we conclude that $H(t) = k$, for some $k \in \mathbb{R}$, and that $S_2(t) = S_1(t) + k \hat G(t)$.
Now, since $P_1(t) \neq P_2(t)$, we have that $k \neq 0$. So, replacing $H(t)$ by $k$ in \eqref{ecu:S}, we obtain that $S(t) = G'(t) + (G(t) \hat G'(t))/\hat G(t)$ as claimed.

Finally, since $S_1(t) \neq S_2(t),$ The opposite is deduced by direct verification using Proposition \ref{prop:inv}.
\end{proof}

\begin{example}\label{ejemplo2}
In order to obtain Example~\ref{ejemplo1}, it suffices to apply Proposition~\ref{prop:param} with 
$G(t) = (\cos(t)+2)(\sin(t)+2)$, $\hat G(t) = \sin(t)+2$, $S_1(t)=\sin(t)+4$, and $k=-1$. 
\end{example}

\section{Darboux first integrals and proof of the main result}

In this section, we use Darboux integrability theory to bound the maximum number of invariant curves that \eqref{eq:Abel} can have without forcing the existence of a center. 

We say that $f(t,x)$, smooth enough and not identically constant, is a first integral of \eqref{eq:Abel} if $\mathcal{X}f=0$, that is, $f(t,x) = 0$ is an invariant curve of \eqref{eq:Abel} with zero cofactor. Equivalently, $f(t,x(t))=0$ is constant if $x(t)$ is a solution of the equation. 

We say that a first integral $f$ of \eqref{eq:Abel} is of Darboux type if \[f(t,x)=\prod_{i=0}^{r}f_i(t,x)^{\alpha_i},\] where $f_i(t,x)=0$ are invariant curves of the equation and $\alpha_i\in\mathbb{C}$.

First, we present Darboux's general result that relates the existence of a first Darboux integral with the linear dependence of the cofactors of the invariant curves. We have adapted its statement to our situation

\begin{theorem}[Darboux's Theorem, \cite{Darboux}]\label{IntegralDarboux}
Let $f_0(t,x)=0,\dots, f_r(t,x)=0$ be invariant curves of \eqref{eq:Abel} with  cofactors $K_0(t,x),\dots, K_r(t,x)$, respectively. If there exist $\alpha_0,\dots,\alpha_r\in\mathbb{C}$ such that $\sum_{i=0}^{r}\alpha_i K_i(t,x)=0$ then $f(t,x)=\prod_{i=0}^{r}f_i(t,x)^{\alpha_i}$ is a first integral of \eqref{eq:Abel}.
\end{theorem}

The following result is a direct application of Theorem \ref{IntegralDarboux} for the case where the invariant curves $f_0(t,x)=0,\dots, f_r(t,x)=0$ are all non-trivial and have the form described in Remark \ref{Rem:invcurve}.

\begin{proposition}\label{prop:Darboux} 
Let $\alpha_i \in \mathbb{R}, i=1,\dots,r$, and $\alpha_0 :=-\sum_{i=1}^r \alpha_i$. If $1-P_i(t)x=0, i=1,\dots,r$ are invariant curves of \eqref{eq:Abel}, then $f(t,x):=x^{\alpha_0}\prod_{i=1}^r (1-P_i(t)x)^{\alpha_i}$ is a first integral of \eqref{eq:Abel} if and only if 
\begin{equation}\label{eq:alphacondition}
\sum_{i=1}^{r}\alpha_i \frac{A(t)}{P_i(t)}=0.
\end{equation}
\end{proposition} 

\begin{proof}
First, we recall that  $f_0(t,x)=x=0$ is always an invariant curve of \eqref{eq:Abel} with cofactor $K_0(t,x)=A(t)x^2+B(t)x$. Furthermore, by Proposition \ref{prop:inv}, we have the cofactor of $1-P_i(t)x=0$ is $K_i(t,x)=A(t)x^2- P_i'(t)x$ for each $i = 1, \ldots, r.$ Therefore,
\begin{align*}
(\mathcal{X}f)(t,x) & =\left(\alpha_0A(t)x^2+\alpha_0B(t)x +\sum_{i=1}^r \alpha_i\Big(A(t)x^2-P'_i(t)x\Big)\right)f(t,x)\\
& = -\left(\sum_{i=1}^r \alpha_i\Big(B(t) + P'_i(t)\Big) x \right)f(t,x).
\end{align*}
Moreover, by Proposition \ref{prop:inv}, $B(t) = -P'_i(t) - \frac{A(t)}{P_i(t)}, i = 1, \ldots, r$. So, we conclude that 
\[(\mathcal{X}f)(t,x)  =  \left(\sum_{i=1}^{r}\alpha_i  \frac{A(t)}{P_i(t)} x \right)f(t,x) \]
Now, by the definition of first integral and Theorem \ref{IntegralDarboux}, our claim follows.
\end{proof}

Note that, as we have seen in the proof of the previous result, \eqref{eq:alphacondition} is a necessary and sufficient condition for the cofactors to be linearly dependent.

We can now complete the proof of the main result (Theorem \ref{th:main}).

\noindent\emph{Proof of Theorem \ref{th:main}.}
If \eqref{eq:Abel} has less than three non-trivial rational limit cycles, then the number of rational limit cycles is bounded by $\deg(A)+1$ because $\deg(A)\geq 1$. Thus, we assume that \eqref{eq:Abel} has $r \geq 3$ non-trivial rational limit cycles, corresponding to the invariant curves $1-P_1(t)x=0,\dots,1-P_r(t)x=0$ of \eqref{eq:Abel} by Proposition \ref{prop:key}. 

Since $r \geq 3$, by Corollary \ref{coro3c} we know that $\deg(P_i)=\deg(A)/2, i=1,\dots,r$, and consequently $\deg(A/P_i)=\deg(A)/2$ for all $i$ by Proposition \ref{prop:inv}. Moreover, by Proposition \ref{prop:Darboux}, the trigonometric polynomials $A/P_i, i = 1, \ldots, r$, are linearly independent. Otherwise, there would be a Darboux first integral and thus a center. 

Finally, since the $\mathbb{R}-$vector space of trigonometric polynomials of degree $\deg(A)/2$ has dimension $\deg(A)+1$, we conclude that $r \leq \deg(A)+1$.\qed

\appendix
\section{On factorizations issues in  the ring of real trigonometric polynomials.}\label{apendice}

It is well known that the ring of real trigonometric polynomials is not a unique factorization domain. However, it is a Dedekind half-factorial domain (\cite[Theorem 3.1]{P-PH}). Therefore, every non-zero non-unit is a finite product of irreducible elements, and any two factorizations into irreducibles of an element in $\mathbb{R}[\cos(t),\sin(t)]$ have the same number of irreducible factors. This allows us to consider the irreducible factors of a given real trigonometric polynomial or to use expressions such us ``$P(t)$ and $Q(t)$ have no common irreducible factors'', regardless of the fact that the greatest common divisor is not defined in half-factorial domains in general.

Recall that, given a non-zero real trigonometric polynomial \[P(t) = \sum_{k=0}^n a_0 \cos(kt) + b_0 \sin(kt),\quad a_k, b_k \in \mathbb{R}\] the degree of $P(t), \deg(P)$, is the biggest $k$ such that $a_k \cdot b_k \neq 0$. Note that $\deg(P Q) = \deg(P) + \deg(Q)$ and, if $\deg(P) > 0$, then $\deg(P') = \deg(P).$ In particular, $P'=0$ if and only if $P \in \mathbb{R}$.

Furthermore, since the irreducible elements of $\mathbb{R}[\cos(t),\sin(t)]$ are those of the form \[a\cos(t)+b\sin(t)+c,\quad a,b,c\in\mathbb{R}, (a,b) \neq (0,0)\] by \cite[Theorem 3.4]{P-PH}, we have that the degree of a non-zero non-unit element of the ring $\mathbb{R}[\cos(t),\sin(t)]$ is the number of its irreducible factors.

Given $z \in \mathbb{R}[\cos(t),\sin(t)]$, in the following we write $\langle z \rangle$ for the principal ideal of the ring $\mathbb{R}[\cos(t),\sin(t)]$ generated by $z$.

The irreducible factors of real trigonometric polynomials without real zeros are characterized by the following proposition.

\begin{proposition}\label{Prop P0}
Let $P(t) \in \mathbb{R}[\cos(t),\sin(t)]$ be non-zero and non-unit. The following statements are equivalent.
\begin{enumerate}
\item $P(t) \neq 0$ for all $t \in \mathbb{R}$.
\item Any irreducible factor of $P(t)$ can be written (up to units) in the form $a\cos(t)+b\sin(t)+c$ where $a,b,c\in\mathbb{R}, (a,b) \neq (0,0)$ and $c^2 > a^2 + b^2$.
\item If $z$ is an irreducible factor of $P(t)$, then $\langle z \rangle$ is a maximal ideal of $\mathbb{R}[\cos(t),\sin(t)]$.
\item Every maximal ideal of $\mathbb{R}[\cos(t),\sin(t)]$ containing $\langle P(t) \rangle$ is principal.
\end{enumerate}
\end{proposition}

\begin{proof}
$(1) \Longleftrightarrow (2).$ Let $z_i \in \mathbb{R}[\cos(t),\sin(t)],\ i = 1, \ldots, n,$ be irreducible elements such that $P = u\, z_1 \cdots z_n$ for some $u \in \mathbb{R}$. Obviously, $P(t) \neq 0$ for all $t \in \mathbb{R}$ if and only if $z_i(t) \neq 0$ for all $t \in \mathbb{R}$ and all $i \in \{1, \ldots, n\}$. Since irreducible elements in $\mathbb{R}[\cos(t),\sin(t)]$ have the form $a\cos(t)+b\sin(t)+c$ with $a,b, c\in\mathbb{R}$ and $(a,b) \neq (0,0)$, and $a\cos(t)+b\sin(t)+c$ has no real zeros if and only if $c^2 > a^2+b^2$, we are done.

\noindent $(2) \Longleftrightarrow (3).$ Let $z = a\cos(t)+b\sin(t)+c \in \mathbb{R}[\cos(t),\sin(t)]$ with $(a,b) \neq (0,0)$. Since, by \cite[Theorem 3.8]{P-PH}, $\langle z \rangle$ is a maximal ideal if and only if $c^2 > a^2 + b^2$, we have the desired equivalence.

\noindent $(3) \Longrightarrow (4).$ Let $\mathfrak{m}$ be a maximal ideal of $\mathbb{R}[\cos(t),\sin(t)]$ such that $P(t) \in \mathfrak{m}$. Since $\mathfrak{m}$ is a prime ideal, there exists an irreducible factor $z$ of $P(t)$ such that $z \in \mathfrak{m}$; equivalently, $\langle z \rangle \subseteq \mathfrak{m}$. From the maximality of $\langle z \rangle$ follows that $\mathfrak{m} = \langle z \rangle$. 

\noindent $(4) \Longrightarrow (3).$ Let $z$ be an irreducible factor of $P(t)$ and let $\mathfrak{m}$ be a maximal ideal of the ring $\mathbb{R}[\cos(t),\sin(t)]$ containing $\langle z \rangle$. Since $\mathfrak{m}$ is principal, there exists $w \in \mathbb{R}[\cos(t),\sin(t)]$ such that $\mathfrak{m} = \langle w \rangle$; in particular, $w$ divides $z$ and the irreducibility of $z$ implies $\langle z \rangle = \mathfrak{m}$.
\end{proof}

\begin{corollary}\label{cor:ufd}
Let $P(t) \in \mathbb{R}[\cos(t),\sin(t)]$ be non-zero and non-unit. If $P(t) \neq 0$ for all $t \in \mathbb{R}$, then $P(t)$ has a unique factorization except for order of factors or product by units.
\end{corollary}

\begin{proof}
Let $u\, z_1 \cdots z_n = v\, w_1 \cdots w_n$ be two factorizations of $P$ into irreducibles, $z_i, w_i,\ i = 1, \ldots, n$, for some $u, v \in \mathbb{R}$. Since, by Proposition \ref{Prop P0}, $\langle w_1 \rangle$ is maximal and $u\, z_1 \cdots z_n = P(t) \in \langle w_1 \rangle$, we have that there exists $j$ such that $z_j \in \langle w_1 \rangle$. So it follows from the irreducibility of $z_j$ that $z_j = u_1 w_1$ for some $u_1 \in \mathbb{R}$. Now it is sufficient to repeat the same argument with $u u_1 \, z_1 \cdots z_{j-1} z_{j+1} \cdots z_n = v\, w_2 \cdots w_n$, and so on, to get the desired result.
\end{proof}

Clearly, the converse of the previous corollary is not true, since there are many real irreducible trigonometric polynomials with real zeros.

\begin{corollary}\label{cor:ufd2}
Let $P(t), H(t)$ and $R(t) \in \mathbb{R}[\cos(t),\sin(t)]$ be non-zero and non-units. If $P(t) \neq 0$, for every $t \in \mathbb{R}, P(t) = H(t) R(t)$ and 
no irreducible factor of $P(t)$ divides $H(t)$, then $P(t)$ divides $R(t)$.
\end{corollary}

\begin{proof}
By Corollary \ref{cor:ufd}, there exist unique irreducible real trigonometric polynomials $z_1, \ldots,$ $z_n$ such that $P(t) = u\, z_1 \cdots z_n$ for some $u \in \mathbb{R}$. If $z_1$ is an irreducible factor of $P(t)$, then $H(t) R(t) = P(t) \in \langle z_1 \rangle$. By Proposition \ref{Prop P0}, $\langle z_1 \rangle$ is maximal. Therefore, $H(t) \in \langle z_1 \rangle$ or $R(t) \in \langle z_1 \rangle$. However, since no irreducible factor of $P(t)$ divides $H(t)$, we conclude that $R(t) \in \langle z_1 \rangle$ and therefore $R(t) = \tilde R(t) z_1$ for some $\tilde R(t) \in \mathbb{R}[\cos(t), \sin(t)]$. Now, if we repeat the same argument with $\tilde P(t) = u\, z_2 \cdots z_n, \tilde R(t)$ and $z_2$, and so on, we get the desired result.
\end{proof}

Now, it is convenient to recall that $\mathbb{C}[\cos(t),\sin(t)]$ is an Euclidean domain (see \cite[Theorem 2.1]{P-PH}. In particular, it is a unique factorization domain.

\begin{corollary}\label{cor:nofactors}
Let $P(t)$ and $Q(t) \in \mathbb{R}[\cos(t),\sin(t)]$ be non-zero and non-units. If $P(t) \neq 0,$ for all $t \in \mathbb{R}$, then $P(t)$ and $Q(t)$ are coprime in $\mathbb{C}[\cos(t),\sin(t)]$ if and only if they have no common irreducible factors in $\mathbb{R}[\cos(t),\sin(t)]$.
\end{corollary}

\begin{proof}
If $P(t)$ and $Q(t)$ have common irreducible factors in $\mathbb{R}[\cos(t),\sin(t)]$, then they have common irreducible factors in $\mathbb{C}[\cos(t),\sin(t)]$. 

Conversely, let us suppose that $P(t)$ and $Q(t)$ have no common irreducible factors in $\mathbb{R}[\cos(t),\sin(t)]$. If $z \in \mathbb{C}[\cos(t),\sin(t)]$ is an irreducible factor of $P(t)$ and $Q(t)$, then $P(t)$ and $Q(t)$ belong to $\langle z \rangle \cap \mathbb{R}[\cos(t),\sin(t)]$. Thus there exists a maximal ideal $\mathfrak{m}$ of $\mathbb{R}[\cos(t),\sin(t)]$ such that $P(t) \in \mathfrak{m}$ and $Q(t) \in \mathfrak{m}$. Since, by Proposition \ref{Prop P0}, $\mathfrak{m}$ is principal, we conclude that, contrary to the hypothesis, $P(t)$ y $Q(t)$ have a real common factor. 
\end{proof}

Note that for the above corollary to hold, the condition $P(t) \neq 0,$ for all $t \in \mathbb{R}$, is mandatory.
\begin{example}
The trigonometric polynomials $P(t) = \sqrt{2}\sin(t) - 1$ and $Q(t) = -\sqrt{2}\cos(t)+1$ are irreducible in $\mathbb{R}[\cos(t),\sin(t)]$ and their respective factorizations in $\mathbb{C}[\cos(t),\sin(t)]$ are 
\[\left(\frac{\left( i-1\right)  \sin{(t)}+\left( i-1\right)  \cos{(t)}-\sqrt{2} i}{2}\right)\left(i \sin{(t)}+\cos{(t)}-\frac{i}{\sqrt{2}}-\frac{1}{\sqrt{2}}\right)\]
and 
\[-\left(\frac{\left(i+1\right)  \sin{(t)}+\left(i-1\right)  \cos{(t)}+\sqrt{2}}{2}\right)\left(i \sin{(t)}+\cos{(t)}-\frac{i}{\sqrt{2}}-\frac{1}{\sqrt{2}}\right).\]
Note that they both have the same last irreducible complex factor.
\end{example}

\section*{Acknowledgements}
The authors are partially supported by Junta de Extremadura/FEDER grant number IB18023. The first and second authors are also partially supported by Junta de Extremadura/FEDER grant number GR21056 and by grant number PID2020-118726GB-I00 funded by MCIN/AEI/ 10.13039/501100011033 and  by ``ERDF A way of making Europe". The third author is also partially supported by Junta de Extremadura/FEDER grant number GR21055 and by grant PID2022-138906NB-C21 funded by MCIN/AEI/10.13039/501100011033, by ``ERDF A way of making Europe".


\begin{thebibliography}{99}

\bibitem{AGG} 
\textsc{M.J. \'Alvarez, A. Gasull, H. Giacomini}, {A new uniqueness criterion for the number of periodic orbits of Abel equations}, J. Differential Equations {\bf 234}(2007), No. 1, 161--176. \doi{10.1016/j.jde.2006.11.004}, \MR{2298969}, \ZB{1118.34036}.

\bibitem{BNP} 
\textsc{D.M. Benardete, V.W. Noonburg, B. Pollina}, {Qualitative tools for studying periodic solutions and bifurcations as applied to the periodically harvested logistic equation}, \textit{Amer. Math. Monthly} {\bf 115}(2008), No. 3, 202--219. \doi{10.1080/00029890.2008.11920518}, \MR{2395031}, \ZB{1361.34038}. 

\bibitem{BCFO}
\textsc{J.L. Bravo, L.A. Calder\'on, M. Fern\'andez, I. Ojeda}, {Rational solutions of Abel differential equations}, \textit{J. Math. Anal. Appl.} {\bf 515}(2022), No. 1, Paper No. 126368, 16 pp. \doi{10.1016/j.jmaa.2022.126368}, \MR{4436497}, \ZB{1503.34076}.

\bibitem{BFY1} 
\textsc{M. Briskin, J.P. Fran\c{c}oise, Y. Yomdin}, {Center conditions, compositions of polynomials and moments on algebraic curves}, \textit{Ergodic Theory Dynam. Systems} {\bf 19}(1999), No. 5, 1201--1220. \doi{10.1017/S0143385799141737}, \MR{1721616}, \ZB{0990.34017}.

\bibitem{BFY2} 
\textsc{M. Briskin, J.P. Fran\c{c}oise, Y. Yomdin}, {Center conditions II: Parametric and model center problems}, \textit{Israel J. Math.} {\bf 118}(2000), 61--82. \doi{10.1007/BF02803516}, \MR{1776076}, \ZB{0989.34021}.

\bibitem{Darboux}
\textsc{G. Darboux}, {Mémoire sur les équations différentielles algébriques du premier ordre et du premier degré} (in French), \textit{Darboux Bull.} (2) II, 60-96 (1878); II, 123-144, 151-200 (1878). \ZB{10.0214.01}

\bibitem{GLL} 
\textsc{A. Gasull, J. Llibre}, {Limit cycles for a class of Abel equations}, \textit{SIAM J. Math. Anal.} {\bf 21}(1990), No. 5, 1235--1244. \doi{10.1137/0521068}, \MR{1062402}, \ZB{0732.34025}.

\bibitem{G} 
\textsc{A. Gasull}, {Some open problems in low dimensional dynamical systems}. \textit{SeMA J.} {\bf 78}(2021), No. 3, 233--269. \doi{10.1007/s40324-021-00244-3}, \MR{4297215}, \ZB{1487.37024}.

\bibitem{GGL13}
\textsc{J. Gin\'e, M. Grau, J. Llibre}, {On the polynomial limit cycles of polynomial differential equations}, \textit{Israel J. Math.} {\bf 181}(2011), 461--475. \doi{10.1007/s11856-011-0019-3}, \MR{2773053}, \ZB{1215.34024}.

\bibitem{GGS16}
\textsc{J. Giné, M. Grau, X. Santallusia}, {The center problem and composition condition for Abel differential equations}, \textit{Expo. Math.} {\bf 34}(2016), No. 2, 210--222. \doi{10.1016/j.exmath.2015.12.002}, \MR{3494281}, \ZB{1351.34030}.

\bibitem{LN} 
\textsc{A. Lins Neto}, {On the number of solutions of the equation $\frac{d x}{dt}=\sum_{j=0}^n a_j(t)x^j$, $0\leq t\leq 1$, for which $x(0)=x(1)$}, 
\textit{Invent. Math.} {\bf 59}(1980), No. 1, 67--76. \doi{10.1007/BF01390315}, \MR{0575082}, \ZB{0448.34012}.

\bibitem{LLWW}	
\textsc{C. Liu, C. Li, X. Wang and J. Wu}, {On the rational limit cycles of Abel equations}, Chaos Solitons Fractals, \textit{Chaos Solitons Fractals} {\bf 110}(2018), 28--32. \doi{10.1016/j.chaos.2018.03.004}, \MR{3790380}, \ZB{1391.34067}.

\bibitem{PSCC} 
\textsc{F. Pakovich}, {Solution of the parametric center problem for the Abel differential equation}, \textit{J. Eur. Math. Soc. (JEMS)} {\bf 19}(2017), No. 8, 2343--2369. \doi{10.4171/JEMS/719}, \MR{3668062}, \ZB{1372.34077}.

\bibitem{P-PH}
\textsc{G. Picavet, M. Picavet-L'Hermite}, {Trigonometric polynomial rings} in \textit{Commutative ring theory and applications (Fez, 2001).
Lecture Notes in Pure and Appl. Math., 231}, Marcel Dekker, Inc., New York, 2003, pp. 419--433. \MR{2029841}, \ZB{1079.13013}.

\bibitem{QSY21}
\textsc{X.Qian, Y.Shen, J.Yang}, {The number of rational solutions of Abel equations,} \textit{J. Appl. Anal. Comput.} {\bf 11}(2021), No. 5, 2535--2552. \doi{10.11948/20200475}, \MR{4321501}.

\bibitem{S}
\textsc{S.Smale}, {Mathematical problems for the next century}, \textit{Math. Intelligencer} {\bf 20}(1998), No. 2, 7--15. \doi{10.1007/BF03025291}, \MR{1631413}, \ZB{0947.01011}.

\bibitem{V22}	    
\textsc{C. Valls},  {Rational solutions of Abel trigonometric polynomial differential equations,} \textit{J. Geom. Phys.} {\bf 180}(2022), Paper No. 104627, 9 pp. \doi{10.1016/j.geomphys.2022.104627}, \MR{4460346}, \ZB{1509.34004}.

\end{thebibliography}
\end{document}